\documentclass[en,oneside,bbfont,article]{amsart} 
\usepackage[lmargin = 30mm, rmargin = 30mm, bmargin = 25mm, tmargin=25mm]{geometry}
\usepackage[utf8]{inputenc}
\usepackage{xcolor}

\usepackage{thmtools}
\usepackage{mathrsfs}
\usepackage{mathtools}
\usepackage{enumitem}
\usepackage{mathabx}
\usepackage{mdframed}
\usepackage{amssymb}
\usepackage{dsfont}
\usepackage{amsthm}
\usepackage{bm}
\usepackage{amsmath,stackrel}
\usepackage{comment}
\usepackage{marginnote}
\usepackage{caption,subcaption}
\usepackage{scalefnt}
\usepackage{float,graphicx,verbatim}
\usepackage{tikz}
\usetikzlibrary{calc}
\usepackage{pgfplots}
\pgfplotsset{compat=1.16}
\RequirePackage[colorlinks=true,linkcolor=black,
	citecolor=blue,urlcolor=blue]{hyperref}

\usepackage{todonotes}
\usepackage{thm-restate}

\usepackage{mathtools}
\mathtoolsset{showonlyrefs}
\numberwithin{equation}{section}

\providecommand{\examplename}{Example}

\newtheorem{theorem}{Theorem}[section]
\newtheorem{proposition}[theorem]{Proposition}
\newtheorem{corollary}[theorem]{Corollary}
\newtheorem{lemma}[theorem]{Lemma}
\theoremstyle{remark}
\newtheorem{remarkx}[theorem]{Remark}
\newenvironment{remark}
    {\pushQED{\qed}\remarkx}
    {\popQED\endremarkx}
\theoremstyle{definition}
\newtheorem*{example*}{\protect\examplename}

\theoremstyle{plain}

\newtheorem*{assumption*}{Assumption}

\usepackage{comment}

\newcommand{\rarrow}{\Rightarrow}
\renewcommand{\phi}{\varphi}
\renewcommand{\epsilon}{\varepsilon}

\newcommand{\ind}{\mathds{1}}

\newcommand{\lp}{\left(}
\newcommand{\rp}{\right)}

\newcommand{\R}{\mathbb{R}}
\newcommand{\N}{\mathbb{N}}

\newcommand{\PP}{\mathbb{P}}
\newcommand{\E}{\mathbb{E}}

\renewcommand*{\H}{\mathcal{H}}
\newcommand{\F}{\mathcal{F}}

\newcommand{\dd}{\mathrm{d}}

\newcommand{\W}{\mathsf{W}}
\newcommand{\X}{\mathsf{X}}
\newcommand{\Y}{\mathsf{Y}}
\newcommand{\U}{\mathsf{U}}
\newcommand{\V}{\mathsf{V}}
\newcommand{\ZZ}{\mathsf{Z}}
\newcommand{\NN}{\mathsf{N}}
\newcommand{\limn}{\lim_{n\to \infty}}
\renewcommand{\L}{\mathsf{L}}

\newcommand{\Hp}{\H^{\otimes p}}
\newcommand{\Ho}{\H^{\odot p}}
\newcommand\cov{\mathrm{Cov}}

\newcommand\dW{\mathrm{W}}
\newcommand\TV{\mathrm{TV}}
\newcommand{\contr}[1]{\widetilde{\otimes}_{#1}}

\newcommand{\inprod}{\langle \cdot, \cdot \rangle}

\newcommand{\dom}{\text{Dom}}

\newcommand{\remove}{\!\!\!\! \!\!\!\! \!\!\!\! \!\!\!\!}

\usepackage[backend=biber,style=ieee,citestyle=numeric,sorting=nyt,maxbibnames=99,dashed=false,sortcites=true]{biblatex}
\renewbibmacro{in:}{}
\appto{\bibsetup}{\sloppy}
\usepackage[foot]{amsaddr}

\title{Fourth-Moment Theorems for Sums of Multiple Integrals}
\author{Andreas Basse-O'Connor$^*$, David Kramer-Bang$^{\dag}$ \& Clement Svendsen$^\ddag$}

\address{$^{*\dag}$Department of Mathematics, Aarhus University, DK}
\address{$^{\ddag}$Department of Computer Science, Aarhus University, DK}

\email{$^*$basse@math.au.dk}
\email{$^\dag$bang@math.au.dk}
\email{$^\ddag$clementks@cs.au.dk}
\begin{document}

\title{Fourth-moment theorems for sums of multiple integrals}

\subjclass[2020]{Primary 60F05, 60H07, 33C45}
\keywords{Central Limit Theorems, Fourth-moment Theorems, Malliavin Calculus, Wiener Chaos}
\date{XX/08/2025}

\commby{}

\begin{abstract}
In this paper, we extend the celebrated fourth-moment theorem of Nualart and Peccati [\emph{Ann. Probab.}~33(1), 2005], which states that convergence of the fourth moment to $3$ implies weak convergence to the standard Gaussian for random variables with zero mean and unit variance in a fixed Wiener chaos. Our main result establishes that the same conclusion holds for random variables expressible as sums of two components in the $p$-th and $q$-th Wiener chaoses, provided $p$ and $q$ have different parities. We also prove quantitative bounds in the $1$-Wasserstein and total variation distances, both scaling with the square root of the fourth cumulant. To the best of our knowledge, this is the first fourth-moment theorem that applies to nontrivial sums across different chaos levels. Finally, we show via a counterexample with $p = 1$, $q = 3$ that the result fails when $p$ and $q$ share the same parity.

\end{abstract}

\maketitle

\section{Introduction and results}\label{sec: intro}
Throughout statistics and probability theory, weak limit theorems play a fundamental role in understanding the asymptotic behavior of complex distributions and stochastic processes. Often, the limiting distribution is a Gaussian distribution, such as in the celebrated classical Central Limit Theorem (CLT). A common technique for proving general weak limit theorems is the \textit{method of moments}, which states the following in the setting of centered Gaussian variables: Let $\NN \sim \mathcal{N}(0,\sigma^2)$ for some $\sigma \in \R$ and $(\X_n)_{n\in \N}$ be a sequence of centered random variables that have moments of all orders. If $\E[\X_n^k] \to \E[\NN^k] \text{ as } n \to \infty$ for all $k\in \N$, then the sequence $(\X_n)_{n\in \N}$ converges in distribution to $\NN$~\cite[~Thm 30.2]{Billingsley}. Although the method of moments can be useful, it is often difficult to compute all moments. In some cases, however, convergence of just the fourth moments suffices. That is, the following implication holds 
\begin{equation}
    \E[\X_n^4] \to 3\sigma^4 \text{ as } n\to \infty \quad \text{ implies } \quad  \X_n \xrightarrow[]{\mathcal{D}} \mathcal{N}(0,\sigma^2) \label{eq: fourth moment condition} \text{ as } n\to \infty.
\end{equation}
We recall that $\E[\NN^4]=3\sigma^4$ for $\NN \sim \mathcal{N}(0,\sigma^2)$.
Theorems guaranteeing~\eqref{eq: fourth moment condition} are known as \textit{fourth-moment theorems}, and they provide a significant improvement of the classical method of moments. It is often convenient to work with cumulants rather than moments, so we denote by $\kappa_4(\X) \coloneqq \E[\X^4]-3\E[\X^2]^2$ the fourth cumulant of a centered random variable $\X$ with finite fourth moment. We can then equivalently state \eqref{eq: fourth moment condition} in terms of the fourth cumulant
\begin{equation}
    \kappa_4(\X_n) \to 0 \text{ as } n\to \infty \quad \text{ implies } \quad  \X_n \xrightarrow[]{\mathcal{D}} \mathcal{N}(0,\sigma^2) \label{eq: fourth cumulant condition} \text{ as } n\to \infty.
\end{equation} 
The first fourth-moment theorem, proved by Nualart \& Pecatti~\cite[Thm~1]{Nualart}, states that~\eqref{eq: fourth moment condition} holds for random variables with only a single term in their \textit{chaos expansions}. A chaos expansion of a square-integrable random variable $\X$, which is measurable with respect to an underlying Gaussian process, is an orthogonal decomposition 
\begin{equation}\label{eq: decomposition}
    \X = \E[\X] + \sum_{p=1}^\infty \mathsf{F}_p,
\end{equation}
where each $\mathsf{F}_p$ is an element of the $p$'th \textit{Wiener chaos} $\mathbb{H}_p$ which is spanned by the $p$'th Hermite polynomial $H_p$ (defined in~\eqref{defn_hermite_poly} below). Equivalently, we can represent $\mathbb{H}_p$ using a Hilbert space isomorphism $I_p:\H^{\odot p} \to \mathbb{H}_p$, where $\H^{\odot p}$ is the space of symmetric $p$-tensors of a suitable Hilbert space $\H$. See Section~\ref{sec: prelims} for more details.

Following the influential result of Nualart and Peccati~\cite[Thm~1]{Nualart}, which established the fourth moment theorem, there has been considerable interest in its extensions. These include settings such as Poisson chaos~\cite{Poisson}, Rademacher chaos~\cite{Rademacher}, free probability~\cite{Free}, and infinitely divisible distributions~\cite{ID}, general homogeneous sums~\cite{MR3500415}, as well as multivariate generalizations~\cite[Thm~6.2.2]{MR2962301} \&~\cite{MR2126978,MR3158721,MR3533710}. Moreover, for any $k \ge 2$, convergence of the $2k$th moment to the $2k$th moment of the Gaussian is also equivalent to weak convergence~\cite{MR3474463}. For an extended survey of results on extensions of the fourth-moment theorem, see~\cite{MR3178608} and the online resource with a broad collection of results related to the Malliavin--Stein method and fourth moment theorems, curated by Ivan Nourdin\footnote{\url{https://sites.google.com/site/malliavinstein/home}}.

As a simple example that fourth-moment theorems do not apply in general, consider a Bernoulli random variable $\Y$ with parameter $p=1/2+1/\sqrt{12}$. It can then be shown that $(\Y-p)/\sqrt{p(1-p)}$ has a mean of $0$, unit variance, and a fourth moment of $3$. However, $\Y$ is evidently not Gaussian. The distribution of $\Y$ can be realized in the Wiener chaos setting as the distribution of $\ind_{(-\infty, t)}(\X)$ for an appropriate $t\in \R$ and $\X \sim \mathcal{N}(0,1)$, and hence $\Y$ has an (infinite) expansion of the form \eqref{eq: decomposition}. While this example shows that~\cite[Thm~1]{Nualart} does not generalize to variables with arbitrary chaos expansions, a natural question to ask is when fourth-moment theorems hold for variables with two or more terms in their chaos expansions. 

Our main result in Theorem~\ref{thm: p,q} shows that a fourth-moment theorem holds if there are exactly two terms in the chaos expansion with orders $p,q \in \N$ of different parities, i.e.\ $p\ne q \mod 2$. To the best of our knowledge, Theorem~\ref{thm: p,q} is the first fourth-moment theorem which applies to random variables not in a fixed Wiener chaos. Furthermore, we obtain a rate of convergence in the $1$-Wasserstein and total variation distances, denoted $d_{\W}$ and $d_{\TV}$ respectively, proportional to the square root of the fourth cumulant, similar to the original result of~\cite{MR2962301}.

\begin{theorem}\label{thm: p,q}
    Let $(\X_n)_{n\in \N}$ be a sequence of random variables with chaos decompositions $\X_n = I_p(u_n)+I_q(v_n)$ for all $n \in \N$,
    where $(u_n)_{n\in \N}$ and $(v_n)_{n\in \N}$ are sequences from $\Ho$ and $\H^{\odot q}$ respectively, and $p,q \in \N$. 
    
    \noindent \textbf{(i)} Assume that $p,q\in \N$ are of different parities. Then a fourth-moment theorem holds. In particular, 
    if $\E[\X_n^2]=\sigma^2 >0$ for all $n \in \N$, then the following are equivalent as $n \to \infty$:
    \begin{itemize}
        \item[{(a)}] $\X_n \xrightarrow[]{\mathcal{D}} \mathcal{N}(0,\sigma^2)$,
        \item[{(b)}] $\emph{Var}[\langle D\X_n, DL^{-1}\X_n\rangle_\H]\to 0$,
        \item[{(c)}] $\|u_n \otimes_{r_1} u_n\|_{\H^{\otimes (2p-2r_1)}}, \|v_n \otimes_{r_2} v_n\|_{\H^{\otimes (2q-2r_2)}} \to 0$ as $n\to \infty$ for all $r_1\in \{1,\dots, p-1\}$ and $r_2 \in  \{1,\dots,q-1\}$,
        \item[{(d)}] $\kappa_4(\X_n) \to 0$.
    \end{itemize}

    Furthermore, for $\NN \sim \mathcal{N}(0,\sigma^2)$ and $d\in \{ d_\dW, d_\TV\}$, there exists a constant $C_{p,q,\sigma}>0$ depending on $p,q$ and $\sigma$ such that 
    \begin{align}
            d(\X_n,\NN) &\le \frac{2}{\sigma \wedge \sigma^2} \emph{Var}[\langle D\X_n, DL^{-1}\X_n\rangle_\H]^{1/2 } \label{eq: ineq1}  \\
    & \le C_{p,q,\sigma}\bigg( \max_{r_1 \in \{1,\dots,p-1\}}\|u_n \otimes_{r_1} u_n\|_{\H^{\otimes (2p-2r_1)}} \label{eq: ineq2} \\
    &\quad + \max_{r_2 \in \{1,\dots,q-1\}}\|v_n \otimes_{r_2} v_n\|_{\H^{\otimes (2q-2r_2)}}  \bigg)^{1/2} \notag \\
    & \le C_{p,q,\sigma}\sqrt{\kappa_4(\X_n) 
    }, \quad \text{ for all }n \in \N. \label{eq: ineq3}
    \end{align}

    \noindent \textbf{(ii)} If $p,q$ are of the same parity, then the fourth-moment theorem does not hold in general. Indeed, there exists $\X=\Y +\mathsf{Z}$ with $\Y \in \mathbb{H}_3$ and $\mathsf{Z} \in \mathbb{H}_1$, such that $\kappa_4(\X)=
    0$ but $\X$ is not Gaussian. 
\end{theorem}

\begin{remark}\label{rem: p,q}
    (I) In Theorem~\ref{thm: p,q}, we can replace the assumption  $\E[\X_n^2]=\sigma^2 $ with just  $\E[\X_n^2]=\sigma^2_n \to \sigma^2 $ as $n\to \infty$ and instead get the bound
    \[
    d(\X_n, \NN) \le C_{p,q,\sigma}\sqrt{\kappa_4(\X_n)
} + d(\NN,\NN_n), \quad \text{ for all }n \in \N,
    \]
    where $d\in \{ d_\dW, d_\TV\}$, $\NN \sim \mathcal{N}(0,\sigma^2)$ and $\NN_n \sim \mathcal{N}(0,\sigma_n^2)$. 
    In both cases, we can estimate the distance between the two Gaussians~\cite[Prop.~3.6.1]{MR2962301} for all $n \in \N$:
    \begin{equation*}
        d_{\text{TV}}(\NN, \NN_n)  \le \frac{2}{\sigma_n^2 \vee \sigma^2}|\sigma_n^2-\sigma^2|, \quad \text{ and } \quad 
        d_{\W}(\NN, \NN_n)  \le \frac{\sqrt{2/\pi}}{\sigma_n\vee \sigma}|\sigma_n^2-\sigma^2|.
    \end{equation*}

(II)  In (ii) of Theorem \ref{thm: p,q}, the variable $\X$ is chosen as a linear combination of Hermite polynomials applied to the components of a bivariate Gaussian vector. Specifically, we find a suitable covariance matrix $\Sigma$ and $a\in \R$ such that for $(\mathsf{U},\mathsf{V})^T \sim \mathcal{N}_2(\mathbf{0}, \Sigma),$ the variable $\X=aH_1(\mathsf{U})+H_3(\mathsf{V})$ satisfies the stated property. 
\end{remark}

To obtain a complete characterization of normality of random variables $\X = \Y + \mathsf{Z}$ where $\Y, \mathsf{Z}$ are in Wiener chaoses with different order parities, we prove that random variables with chaos decompositions of the type considered in Theorem~\ref{thm: p,q} cannot themselves be Gaussian. Theorem~\ref{thm: ikke gaussisk1} therefore generalizes~\cite[Cor.~5.2.11]{MR2962301}, which asserts that a random variable in a fixed Wiener chaos of order $p \ge 2$ has strictly positive fourth cumulant and hence cannot be Gaussian. 

\begin{theorem}\label{thm: ikke gaussisk1}
    Let $\X$ be a non-degenerate random variable with chaos decomposition $\X=\Y+\mathsf{Z}$, where $\Y \in \mathbb{H}_p$ and $\mathsf{Z}\in \mathbb{H}_q$ for $p,q \in \N$ with different parities. Then $\kappa_4(\X)>0$, and in particular $\X$ cannot be Gaussian.  
\end{theorem}

From Theorem~\ref{thm: p,q} (ii), we note that there exist $\X=\Y +\mathsf{Z}$ with $\Y \in \mathbb{H}_3$ and $\mathsf{Z} \in \mathbb{H}_1$, such that $\kappa_4(\X)=0$ but $\X$ is not Gaussian. One might expect that the same construction could provide counter examples to the fourth-moment theorem when the terms in the decomposition are instead in $\mathbb{H}_1$ and $\mathbb{H}_q$ respectively for general odd $q$. However, this is not the case, as seen in Proposition~\ref{prop:modeks_p_modeks} below, where in fact we show that the fourth cumulant is strictly positive for all variables of the type considered in Theorem \ref{thm: p,q} (ii) (and Remark \ref{rem: p,q} (II)). Hence, the example from Theorem~\ref{thm: p,q} (ii) is not easily generalized.

\begin{proposition}\label{prop:modeks_p_modeks}
    Let $\bm{\Sigma}$ be a covariance matrix with  $\bm{\Sigma}_{1,1}=\bm{\Sigma}_{2,2}=1$ and let $a\in \R$. For
    $(\mathsf{U},\mathsf{V})^T\sim \mathcal{N}_2(\bm{0},\bm{\Sigma})$ it holds that $\kappa_4(aH_1(\mathsf{U})+H_5(\mathsf{V}))>0$.
\end{proposition}
Comparing Proposition~\ref{prop:modeks_p_modeks} with Remark \ref{rem: p,q} (II), we see that it is impossible to find an example of a non-Gaussian variable with vanishing fourth cumulant by using the same strategy as in Theorem~\ref{thm: p,q}(ii) when $\ZZ \in \mathbb{H}_1$ and $\Y \in \mathbb{H}_q$ for general odd $q$. 

Our next result shows that by imposing  stronger conditions, it is possible to get a fourth-moment theorem for variables with infinite chaos expansions. Specifically, we require independence of all terms in the expansions and an additional integrability assumption in terms of the \textit{Ornstein-Uhlenbeck operator} $L$, which will be defined in more detail in Section~\ref{sec: prelims}.
\begin{theorem}\label{thm: sum uafh.}
        For each $p\in \N$, consider a sequence $(\mathsf{F}_{p,n})_{n\in \N}$ from $\mathbb{H}_p$ such that for all $n\in\N$, the following holds:
        \begin{enumerate}
            \item The variables $\mathsf{F}_{1,n},\mathsf{F}_{2,n},\dots$, are independent.
            \item The sum $\sum_{p=1}^\infty \mathsf{F}_{p,n}$ converges in $\L^{\!4}(\Omega)$ to a limit $\X_n$ with $\E[\hspace{0.3mm}\mathsf{X}_n^2]=\sigma^2$. 
            \item The variables $\mathsf{X}_n$ are in $\emph{\dom}(L)$ and $\sup_{n\in \N} \|L \mathsf{X}_n\|_{\mathsf{L}^{\!2}(\Omega)}^2 < \infty$.
        \end{enumerate} 
        If $\limn \E[\hspace{0.3mm}\mathsf{X}_n^4]=3\sigma^4$, then $\mathsf{X}_n\xrightarrow[]{\mathcal{D}} \mathcal{N}(0,\sigma^2)$ as $n\to \infty$.
\end{theorem}
 As a corollary to Theorem~\ref{thm: sum uafh.}, we obtain a fourth-moment theorem for variables with finite chaos decompositions consisting of independent terms, without requiring any additional regularity. 

\begin{corollary}\label{cor: cor}
    Let $M\in \N$ and consider for each $p\in \{1,\dots, M\}$ a sequence $(\mathsf{F}_{p,n})_{n\in \N}$ from $\mathbb{H}_p$ such that for all $n \in \N$, the variables $\mathsf{F}_{1,n},\dots, \mathsf{F}_{M,n}$ are independent. Let $\mathsf{X}_n = \sum_{p=1}^M \mathsf{F}_{p,n}$ and assume $\E[\hspace{0.3mm}\mathsf{X}_n^2]=\sigma^2$ for all $n\in \N$. Then $\limn \E[\hspace{0.3mm}\mathsf{X}_n^4]=3\sigma^4$ if and only if $\mathsf{X}_n\xrightarrow[]{\mathcal{D}} \mathcal{N}(0,\sigma^2)$ as $n\to \infty$.
\end{corollary}

In summary, while determining when fourth moment theorems hold for random variables with arbitrary chaos expansions remains an open problem, we provide here the first steps towards a deeper understanding of this problem. In the case of two multiple integrals with orders $p,q\in \N$, we show that a fourth moment theorem holds in ``half'' of the cases, i.e.\ when $p$ and $q$ have different parities. Conversely, we provide a counterexample for $p=1$, $q=3$ that shows that the fourth moment theorem does not generalize to $p$ and $q$ both odd. However, Proposition~\ref{prop:modeks_p_modeks} shows that our approach for finding this counterexample fails even for $p=1$, $q=5$. Hence, it remains an open question whether the fourth moment theorem fails for all $p$ and $q$ with the same parity, or if there are other combinatorial assumptions one can impose on the orders. Finally, in a fixed Wiener chaos, it is shown in ~\cite{Optimal} that the optimal rate is $\max\{\left|\E[\mathsf{F}^3]\right|, \kappa_4(\mathsf{F})\}$. However, their proof does not generalize to our setting due to edge cases arising from the presence of different order multiple integrals. In particular, handling contractions like $(u \contr{r_1} v)\contr{r_2} u$ turns out to be difficult when $r_1=p<q$, a situation which cannot occur in the fixed Wiener chaos setting. 

We now briefly describe the main ideas in our proofs. The main technical lemma which we will use repeatedly in the the proofs of Theorems~\ref{thm: p,q} and~\ref{thm: ikke gaussisk1} shows that the fourth cumulant is particularly simple when $p$ and $q$ are of different parities - it decomposes into the sum of the fourth cumulants of each term and then a covariance term between squares of multiple integrals, which turns out to be easy to handle. 

\begin{lemma}\label{lem: key lemma}
    Let $\X$ be a random variable with chaos decomposition $\X=\Y+\mathsf{Z}$, where $\Y \in \mathbb{H}_p$ and $\mathsf{Z}\in \mathbb{H}_q$ for $p,q \in \N$ with different parities. Then
    \[
    \kappa_4(\X)=\kappa_4(\Y)+\kappa_4(\mathsf{Z}) + 6 \cov(\Y^2,\mathsf{Z}^2), \quad \text{ and } \quad \kappa_4(\Y), \kappa_4(\ZZ) \le \kappa_4(\X).
    \]
\end{lemma}

Theorem~\ref{thm: p,q} is proven by using the fact that $\E[|\langle D\X_n, DL^{-1}\X_n\rangle_\H-\sigma^2|^2]^{1/2}$ upper bounds the total variation and Wasserstein distances up to constants. Using $\X_n=I_p(u_n)+I_q(v_n)$ and bilinearity of the inner product yields four terms, two of which contain just $I_p(u_n)$ and $I_q(v_n)$ respectively, and the other two terms contain both. The non-mixed terms are easily handled, and for the mixed terms, the critical tool is Lemma~\ref{lem: key lemma} and a bound based on the product formula. 

Theorem~\ref{thm: ikke gaussisk1} is proven by arguing that $\X$ has strictly positive fourth cumulant, which again will follow from Lemma~\ref{lem: key lemma}. 
In the proof of Theorem~\ref{thm: sum uafh.}, we utilize the fact that the $m$'th cumulant of a sum is the sum of $m$'th cumulants when the terms are independent. The proof is concluded by showing pointwise convergence of the associated characteristic functions.

The rest of the paper is structured as follows. In Section~\ref{sec: prelims}, we briefly introduce the necessary notation and concepts which we will need. In Section~\ref{sec: proofs}, we prove Theorems~\ref{thm: p,q},~\ref{thm: ikke gaussisk1} \&~\ref{thm: sum uafh.}, as well as Proposition~\ref{prop:modeks_p_modeks}, Corollary~\ref{cor: cor}, and Lemma~\ref{lem: key lemma}.

\section{Preliminaries} \label{sec: prelims} 
We use the following notation throughout: Let $\N=\{1,2,\dots,\}, \, \N_0 = \N \cup \{0\}$. For random variables $(\X_n)_{n\in \N}, \, \X$, we write $\X_n \xrightarrow[]{\mathcal{D}} \X$ as $n\to \infty$ if $\X_n$ converges in distribution to $\X$. Similarly, if $\X\sim \mu$ for some probability measure $\mu$, we write $\X_n \xrightarrow[]{\mathcal{D}} \mu$. For any $p\in \N$, the space $\L^{\!p}(\Omega)$ denotes the set of random variables $\X$ with $\E[|\X|^p]<\infty$. For a random variable $\X \in \L^{\!4}(\Omega)$, recall that $\kappa_4(\X) \coloneqq \E[\X^4] - 3\E[\X^2]^2$ denotes the fourth cumulant. Finally, let $\vee, \wedge$ denote the maximum and minimum operators respectively.

In this section, we introduce the key concepts required for the proofs. For a detailed exposition, see~\cite{MR2962301}.
Fix a probability space $(\Omega,\F,\PP)$, a real, separable Hilbert space $(\H,\inprod_\H)$ and an isonormal Gaussian process $\W=(\W(h))_{h\in \H}$ and assume that $\F$ is the  $\sigma$-algebra generated by $\W$. That is, $\W$ is a stochastic process on $\H$ such that $(\W(h_1),\dots, \W(h_n))$ is a centered multivariate Gaussian vector in $\R^n$ for all $h_1,\dots,h_n \in \H$ and $n\in \N$ and $\W$ is a linear isometry from $\H$ into $\L^{\!2}(\Omega)$. An essential tool in Gaussian analysis is the class of Hermite polynomials. For $p\in \N$, the $p$'th Hermite polynomial $H_p:\R \to \R$ is given by \begin{equation}\label{defn_hermite_poly}
    H_p(x) \coloneqq (-1)^p \exp(x^2/2) \frac{\dd^p}{\dd x^p}\exp(-x^2/2), \, \text{ and } \, H_0(x) \coloneqq 1, \, \text{ for all }x\in \R.
\end{equation} The Hermite polynomials give rise to an important family of subspaces of $\L^{\!2}(\Omega)$. For each $p\in \N_0$, the closed linear subspace $\mathbb{H}_p = \overline{\text{Span}}\{H_p(\mathsf{W}(h)) : h\in \H, \, \|h\|_\H = 1\}$ is called the $p$'th Wiener chaos. The Wiener chaos spaces provide the following useful orthogonal decomposition $\mathsf{L}^{\!2}(\Omega) = \bigoplus_{p=0}^\infty \mathbb{H}_p$~\cite[Thm~2.2.4]{MR2962301}.  
Alternatively, we can represent Wiener chaoses as the images under certain linear operators. To define these, we let $\Hp$ denote the $p$-fold tensor product of $\H$ with itself and $\Ho\subset \Hp$ the corresponding symmetric tensor product for $p\in \N_0$, where by convention $\H^{\otimes 0} = \H^{\odot 0} = \R$. Then, there is a unique Hilbert space isomorphism $I_p : \Ho \to \mathbb{H}_p$ satisfying $I_p(u^{\otimes p})=H_p(W(u))$ for all unit vectors $u\in \H$, and where $\Ho$ is equipped with the norm $\|\cdot\|_{\Ho} = \|\cdot\|_{\Hp}/\sqrt{p!}$~\cite[Thm~2.7.7]{MR2962301}. We refer to elements of the form $I_p(u)$ as $p$'th order multiple integrals. 
We may now rephrase the orthogonal decomposition of $\L^{\!2}(\Omega)$ stated in~\eqref{eq: decomposition} in terms of multiple integrals: for $\X \in \L^{\!2}(\Omega)$ there are vectors $u_p \in \H^{\odot p}$ such that
\begin{equation}
    \X = \E[\X] +  \sum_{p=1}^{\infty} I_p(u_p). \label{eq: Chaos}
\end{equation}
Equation \eqref{eq: Chaos} then allows us to infer properties of general random variables by understanding multiple integrals.
Another pivotal tool is the following \textit{product formula}, which asserts that a product of multiple integrals is a sum of other multiple integrals ~\cite[Thm 2.7.10]{MR2962301}. Let $p,q\in \N$ and $u\in \Ho, v\in \H^{\odot q}$. Then,
\begin{equation}
    I_p(u)I_q(v) = \sum_{r=0}^{p\wedge q}r! {q \choose r}{p \choose r} I_{p+q-2r}\lp u\,\widetilde{\otimes}_{r}\,v\rp, \label{eq: product formula}
\end{equation}
    where $u \, \widetilde{\otimes}_r \, v$ denotes the symmetrization of the $r$-contraction of $u$ and $v$ (see~\cite[App.~B]{MR2962301}). Define the Ornstein-Uhlenbeck operator $L:\dom(L) \to \L^{\!2}(\Omega)$ by $L\X = \sum_{p=1}^\infty pI_p(u_p)$, with
    \[
    \dom(L) \coloneqq \left\{\X = \E[\X] + \sum_{p=1}^\infty I_p(u_p) \in \L^{\!2}(\Omega)\, \bigg|\, \sum_{p=1}^\infty p^2\E[I_p(u_p)^2] < \infty \right\}
    \]
and its pseudo inverse $L^{-1}:\L^{\!2}(\Omega)\to \L^{\!2}(\Omega)$ by $L^{-1} I_p(u) = - p^{-1}I_{p}(u)$ for any $p\in \N_0, u\in \Ho$ and extends to arbitrary elements by using the chaos decomposition and linearity. 

Finally, we briefly introduce the Malliavin derivative and the divergence operator. Let $d\in \N$ and $\phi:\R^d \to \R$ be infinitely differentiable and such that all derivative are of polynomial growth. Then $\mathsf{F}=\phi(\W(h_1),\dots,\W(h_d))$ is said to be a \textit{smooth} random variable for any $h_1,\dots,h_d\in \H$. 
For such a smooth random variable, we define the Malliavin derivative of $\mathsf{F}$ to be the $\H$-valued random element
\[
D \mathsf{F} = \sum_{i=1}^d \frac{\partial}{\partial x_i}\phi(\W(h_1),\dots,\W(h_d))h_i.
\]
We will make use of the facts that $D$ is a linear operator and that multiple integrals are always smooth random variables with Malliavin derivative $DI_p(u)=pI_{p-1}(u)$ for $p\in \N, u\in \Ho$. The divergence operator $\delta$ is the adjoint of $D$. In particular, it satisfies the identity $\E[\langle D I_p(u), \mathsf{F}h  \rangle_\H] = \E[I_p(u) \delta (\mathsf{F}h)]$ for any $p\in \N, u\in \Ho, \mathsf{F}\in \L^{\!2}(\Omega)$ and $h\in \H$. 

\section{Proofs}\label{sec: proofs}
We are now ready to prove the results stated in Section~\ref{sec: intro}.

\begin{proof}[Proof of Lemma~\ref{lem: key lemma}]
    Due to the correspondence between multiple integrals and Wiener chaoses, there exist vectors $u \in \H^{\odot p}, v \in \H^{\odot q}$  for each $n\in \N$ such that $\X = I_p(u)+I_q(v)$. Hence, need to show that 
        \begin{equation}
            \kappa_4(\X) = \kappa_4(I_p(u))+\kappa_4(I_p(v)) + 6\cov(I_p(u)^2, I_q(v)^2). \label{eq: fjerede kumulant for sum}
        \end{equation}
        By the binomial theorem, it follows that
        \begin{equation}
            \E[\X^4] = \sum_{k=0}^4 {4\choose k} \E[I_p(u)^4I_q(v)^{4-k}]. \label{eq: binomial}
        \end{equation}
        For $k=1,3$, we now show that corresponding mixed moments in~\eqref{eq: binomial} vanish. Indeed, using the product formula~\eqref{eq: product formula} three times yields
        \begin{align}\label{eq: grimt udtryk}
            \begin{aligned}
                        &I_p(u)^3 I_q(v) \\
                        & = \sum_{r_1=0}^{p}\sum_{r_2=0}^{(2p-2r_1) \wedge p}\, \, \sum_{r_3=0}^{(3p-2(r_1+r_2))\wedge q}\remove \, \,\,  c(r_1,r_2,r_3,p,q)I_{3p+q-2(r_1+r_2+r_3)}(h(r_1,r_2,r_3,u,v)) .
        \end{aligned}
        \end{align}
         for some constants $c(r_1,r_2,r_3,p,q)\in \R$ and vectors 
         \[
         h(r_1,r_2,r_3,u,v)\in \H^{\odot(3p-q-2(r_1+r_2+r_3))}.
         \]The constants and vectors can be given explicitly, but as they are unimportant for this proof, we choose not to state them to simplify the expression. Recall now that multiple integrals of order $m\in \N_0$ are centered unless $m=0$. Since $p$ and $q$ have different parities, $3p+q$ is odd, which means there are no choices of $r_1,r_2,r_3\in \N_0$ such that $3p+q-2(r_1+r_2+r_3) = 0$. Therefore, all terms in~\eqref{eq: grimt udtryk} have expectation $0$, implying that $\E[I_p(u_n)^3 I_q(v_n)]=0$ for all $n \in \N$.
         Analogous arguments show that $\E[I_p(u_n) I_q(v_n)^3]=0$.
         Thus,
         \begin{equation}
             \E[\X^4] = E[I_p(u)^4] + E[I_q(v)^{4}] + 6E[I_p(u)^2I_q(v)^{2}].
             \label{eq: moment_for_sum}
         \end{equation}
         Subtracting $3\E[\X^2]^2=3\lp\E[I_p(u)^2]^2+\E[I_q(v)^2]^2 +2\E[I_p(u)^2]\E[I_q(v)^2]\rp$
         on both sides of~\eqref{eq: moment_for_sum} shows~\eqref{eq: fjerede kumulant for sum}. By \cite[Eqs~(6) \&~(12)]{MR1048936}, the term $\cov(I_p(u)^2,I_q(v)^2)$ is non-negative and by \cite[Cor. 5.2.11]{MR2962301}, $\kappa_4(I_q(v))>0$. Combining these facts with \eqref{eq: fjerede kumulant for sum} shows that
         \[
          \kappa_4(\X) = \kappa_4(I_p(u))+\kappa_4(I_q(v)) + 6\cov(I_p(u)^2, I_q(v)^2) \ge \kappa_4(I_p(u)).
         \]
         An identical argument shows the claim for $I_q(v)$.
\end{proof}

    \begin{proof}[Proof of Theorem~\ref{thm: p,q}]
    \textbf{Part~(i)}. We start by showing the inequalities \eqref{eq: ineq1}, \eqref{eq: ineq2} and \eqref{eq: ineq3}.
    Assume that $\NN \sim \mathcal{N}(0,\sigma^2)$. From~\cite[Prop. 5.1.3]{MR2962301}, we have 
    \begin{equation}
        d_{\W}(\X_n,\NN ) \le \frac{1}{\sigma}\E[|\langle D\X_n, DL^{-1}\X_n\rangle_{\H}-\sigma^2|^2]^{1/2},\quad \text{ for all }n \in \N. \label{eq: Wasserstein}
        \end{equation}
    Again, due to ~\cite[Prop. 5.1.3]{MR2962301}, the bound is the same for the total variation distance, except the factor in front is $2/\sigma^2$ instead of $1/\sigma$. Hence, for $d\in \{d_\W, d_{\TV}\}$, we have $$d(\X_n,\NN)\le \frac{2}{\sigma \wedge \sigma^2}\E[|\langle D\X_n, DL^{-1}\X_n\rangle_{\H}-\sigma^2|^2]^{1/2}.$$ 
    To get \eqref{eq: ineq1} in Theorem~\ref{thm: p,q}, we now show that 
    \begin{equation}\label{eq_innner_prod_var}
        \E[|\langle D\X_n, DL^{-1}\X_n\rangle_{\H}-\sigma^2|^2] = \text{Var}[\langle D\X_n, DL^{-1}\X_n\rangle_{\H}].
    \end{equation}
    Indeed, using basic properties of the various operators, we get
    \[
    \E[\langle D\X_n, DL^{-1}\X_n\rangle_{\H}] = -\E[ \X_n (\delta D L^{-1} \X_n)] = \E[\X_n (L L^{-1}\X_n)] = \E[\X_n^2] = \sigma^2,
    \]
    which shows~\eqref{eq_innner_prod_var}. Now, writing out the inner product and using the definition of $L^{-1}$, we obtain
    \begin{align}
    \begin{aligned}\label{eq: inner product}
        &\langle D\X_n, DL^{-1}\X_n\rangle_{\H} - \sigma^2 \\
         &= p^{-1}\langle DI_p(u_n),DI_p(u_n) \rangle_\H 
         - \sigma_{n}^2 + q^{-1}\langle DI_q(v_n),DI_q(v_n) \rangle_\H - \tau_{n}^2\\ &\quad  + \lp p^{-1}+q^{-1}\rp  \langle DI_p(u_n),DI_q(v_n) \rangle_\H, 
    \end{aligned}
    \end{align}
    where $\sigma_n^2 = \E[I_p(u_n)^2], \tau_n^2=\E[I_q(v_n)^2]$ and $\sigma_n^2+\tau_n^2=\sigma^2$ for all $n\in \N$ by orthogonality. Plugging~\eqref{eq: inner product} into~\eqref{eq: Wasserstein}, and using the triangle inequality, we get
    \begin{align}
        \E[|\langle D\X_n, DL^{-1}\X_n\rangle_{\H}-\sigma^2|^2]^{1/2} & \le \E\left[\left|p^{-1}\|DI_p(u_n)\|_\H^2 - \sigma_{n}^2\right|^2\right]^{1/2} \\
        & \quad + \E\left[\left|q^{-1}\|DI_q(v_n)\|_\H^2 - \tau_{n}^2\right|^2\right]^{1/2} \label{eq: vurderet1} \\ 
        & \quad +  \lp p^{-1}+q^{-1}\rp \E[|\langle  DI_p(u_n),DI_q(v_n)\rangle_\H|^2]^{1/2} . \label{eq: vurderet2}
    \end{align}
    As a consequence of~\cite[Lem. 5.2.4]{MR2962301}, there is a constant $R_{p,q}>0$ such that
    \begin{align}
    \begin{aligned}\label{eq: contract 1}
                 & \E\left[\left|p^{-1}\|DI_p(u_n)\|_\H^2 - \sigma_{n}^2\right|^2\right] \le R_{p,q}\max_{r_1 \in \{1,\dots,p-1\}} \|u_n \otimes_{r_1} u_n\|_{\H^{\otimes (2p-2r_1)}},\\
         &\E\left[\left|q^{-1}\|DI_q(v_n)\|_\H^2 - \tau_{n}^2\right|^2\right] \le R_{p,q}\max_{r_2 \in \{1,\dots,q-1\}} \|v_n \otimes_{r_2} v_n\|_{\H^{\otimes (2q-2r_2)}}.
    \end{aligned}
    \end{align}
    Consequently, both expectations of the left-hand sides in \eqref{eq: contract 1} are bounded by
    \[
    R_{p,q}\lp \max_{r_1 \in \{1,\dots,p-1\}} \|u_n \otimes_{r_1} u_n\|_{\H^{\otimes (2p-2r_1)}} + \max_{r_2 \in \{1,\dots,q-1\}} \|v_n \otimes_{r_2} v_n\|_{\H^{\otimes (2q-2r_2)}}\rp.
    \]
   
    What remains is thus to estimate the mixed term. Now, assume without loss of generality that $p<q$. Then, by Lemma~\cite[Lem. 6.2.1]{MR2962301},
    \begin{align}
    \begin{aligned}\label{eq: contract 3}
                 & \E\big[|q^{-1}\langle DI_p(u_n),DI_q(v_n)\rangle_\H|^2\big]  \le p!^2 {q-1\choose p-1}^2(q-p)!\|u_n\|_{\Hp}^2 \|v_n \otimes_{q-p} v_n\|_{\H^{\otimes 2p}}\\
         & \qquad + \frac{p^2}{2}\sum_{r=1}^{p-1}(r-1)!^2{p-1\choose r-1}^2{q-1\choose r-1}^2(p+q-2r)! \\
         & \qquad \times (\|u_n \otimes_{p-r} u_n\|^2_{\H^{\otimes 2r}}+\|v_n \otimes_{p-r} v_n\|^2_{\H_{^{\otimes 2r}}}) \\
         & \le K_{p,q}\lp \max_{r_1 \in \{1,\dots,p-1\}} \|u_n \otimes_{r_1}  u_n\|_{\H^{\otimes (2p-2r_1)}} + \max_{r_2 \in \{1,\dots,q-1\}}\|v_n \otimes_{r_2} v_n\|_{\H^{\otimes (2q-2r_2)}}  \rp \\
         & \qquad \times (1+\|u_n\|^2_{\H^{\otimes p}})
         \\
         & \le K_{p,q}\lp \max_{r_1 \in \{1,\dots,p-1\}} \|u_n \otimes_{r_1}  u_n\|_{\H^{\otimes (2p-2r_1)}} + \max_{r_2 \in \{1,\dots,q-1\}}\|v_n \otimes_{r_2} v_n\|_{\H^{\otimes (2q-2r_2)}}  \rp\\
         & \qquad \times (1+\sigma^2), 
    \end{aligned}
    \end{align}
     for some constant $K_{p,q}>0$, and using in the final inequality that $\|u_n\|^2_{\Hp}\le \sigma^2$ by the assumption that $\E[\X_n^2]=\sigma^2$ for all $n \in \N$ and since $I_p$ is an isometry. Similarly, for all $n\in \N$, it holds that
    \begin{align}\label{eq: contract 4}
        \begin{aligned}
         & \E\big[|p^{-1}\langle DI_p(u_n),DI_q(v_n)\rangle_\H|^2\big]\le \frac{q^2}{p^2} K_{p,q}(1+\sigma^2)  \\ 
         & \quad \times \lp \max_{r_1 \in \{1,\dots,p-1\}} \|u_n \otimes_{r_1}  u_n\|_{\H^{\otimes (2p-2r_1)}} + \max_{r_2 \in \{1,\dots,q-1\}}\|v_n \otimes_{r_2} v_n\|_{\H^{\otimes (2q-2r_2)}}  \rp.  
     \end{aligned}
    \end{align}
     Putting together the inequalities~\eqref{eq: contract 1},~\eqref{eq: contract 3} and~\eqref{eq: contract 4} proves the inequality \eqref{eq: ineq2}. 
    To show the final inequality \eqref{eq: ineq3} we use~\cite[Lem.~5.2.4]{MR2962301} and Lemma~\ref{lem: key lemma} to get
    \begin{align}
         &\bigg( \max_{r_1 \in \{1,\dots,p-1\}} \|u_n \otimes_{r_1}  u_n\|_{\H^{\otimes (2p-2r_1)}} + \max_{r_2 \in \{1,\dots,q-1\}}\|v_n \otimes_{r_2} v_n\|_{\H^{\otimes (2q-2r_2)}}  \bigg)(1+\sigma^2) \nonumber \\
          &\quad \le (1+\sigma^2)(\kappa_4(I_p(u_n))+\kappa_4(I_q(v_n))) \le (1+\sigma^2)\kappa_4(\X_n). \label{eq: ulighed, cumulant}
    \end{align}
    Inserting~\eqref{eq: ulighed, cumulant} into~\eqref{eq: contract 1},~\eqref{eq: contract 3} and~\eqref{eq: contract 4} yields 
    \begin{equation*}
         \E[|\langle D\X_n, DL^{-1}\X_n\rangle_{\H}-\sigma^2|^2]^{1/2}  \le  2\sqrt{R_{p,q}\kappa_4(\X_n)} + (1+q/p)\sqrt{(1+\sigma^2)K_{p,q} \kappa_4(\X_n)}.
    \end{equation*}
    The proof of \eqref{eq: ineq3} is thus completed by choosing\ $$C_{p,q,\sigma}=\frac{2\sqrt{1+\sigma^2}}{\sigma \wedge \sigma^2}\max\left\{\sqrt{R_{p,q}},\sqrt{K_{p,q}}(1+q/p)\right\}.$$ The chain of inequalities \eqref{eq: ineq1}, \eqref{eq: ineq2} and \eqref{eq: ineq3} shows that $(d) \rarrow (c) \rarrow (b) \rarrow (a)$. Hence, we get equivalence of $(a)$-$(d)$ by showing $(a) \rarrow (d)$. Therefore, assume that $\X_n \xrightarrow[]{\mathcal{D}} \mathcal{N}(0,\sigma^2)$. By~\cite[Thm 3.5]{billing2} it suffices to show that $(\X_n^4)_{n\in \N}$ is uniformly integrable, which in turn follows if there exists some $k >4$ such that $\sup_{n\in \N} \E[|\X_n|^k]^{1/k}<\infty$ (by~\cite[p. 218]{Billingsley}). By \textit{hyper contractivity} of multiple integrals (\cite[Thm 2.7.2]{MR2962301}), there are constants $c_p, c_q>0$ such that
    \begin{align}
        &\sup_{n\in \N} \E[I_p(u_n)^6]^{1/6} \le c_p\sup_{n\in \N}\E[I_p(u_n)^2]^{1/2} \le c_p \sigma^2 < \infty, \\
        &\sup_{n\in \N} \E[I_q(v_n)^6]^{1/6} \le c_q\sup_{n\in \N}\E[I_q(v_n)^2]^{1/2} \le c_q \sigma^2 < \infty.
    \end{align}
    Then, Minkowski's inequality~\cite[p. 242]{Billingsley} concludes the proof:
     \[
     \sup_{n\in \N}\E[\X_n^6]^{1/6} \le \sup_{n\in \N} \E[I_p(u_n)^6]^{1/6}+ \sup_{n\in \N} \E[I_q(v_n)^6]^{1/6} < \infty. 
     \]

     \textbf{Part~(ii)}. Let $(\U,\V)^\intercal \sim \mathcal{N}_2\left(\bm{0},\bm{\Sigma}\right)$ be a bivariate normal random vector with mean $\bm{0}$ and covariance matrix $\bm{\Sigma}$, where $\bm{\Sigma}_{1,1}=\bm{\Sigma}_{2,2}=1$ and $\bm{\Sigma}_{1,2}=\bm{\Sigma}_{2,1}=\rho \in [-1,1]$. Define  $\ZZ \coloneqq 10 \U \in \mathbb{H}_1$, $\Y \coloneqq \V^3-3\V \in \mathbb{H}_3$ and $\X \coloneqq \Y + \ZZ = 10\U+\V^3-3\V$. In the following, we show that there exists a $\rho \in [-1,1]$ such that $\kappa_4(\X)=0$, but that $\X$ is not Gaussian for this $\rho$, i.e. that $\E[\X^6] \ne 15 \E[\X^2]^3$. To see that $\kappa_4(\X)=\E[\X^4]-3\E[\X^2]^2=0$, note that
\begin{equation*}
    \E[\X^2]=\E[(10\U)^2+(\V^3-3\V)^2+20\U(\V^3-3\V)]=10^2\E[\U^2]+\E[(\V^3-3\V)^2]=106,
\end{equation*} since $\E[(\V^3-3\V)^2]=
3!$ and $\E[\U(\V^3-3\V)]=0$ by~\cite[Prop.~2.2.1]{MR2962301}. Hence, we have $3\E[\X^2]^2=33708$. Next, to calculate $\E[\X^4]$, we note the following useful property for bivariate normal distributions. Let $f_{(\U,\V)}:\R\times \R \to (0,\infty)$ be the density function of $(\U,\V)^\intercal$, where
\begin{equation*}
    f_{(\U,\V)}(x,y)=\frac{1}{2\pi \sqrt{1-\rho^2}} e^{-\frac{(x^2+y^2-2\rho xy)}{2(1-\rho^2)}}=\frac{1}{\sqrt{2\pi}}e^{-x^2/2} \frac{1}{\sqrt{2\pi (1-\rho^2)}}e^{-\frac{(y-\rho x)^2}{2(1-\rho^2)}},
\end{equation*} for $x,y \in \R$. For the remainder of the proof, denote by $\gamma(\dd x)= (2\pi)^{-1/2}e^{-x^2/2}\dd x$ the standard Gaussian measure on $\R$. Hence, for all $n,m \in \N$, 
\begin{equation}\label{eq:mixed_moments}
    \E[\U^n\V^m]=\int_\R x^n 
    \int_\R y^m \frac{1}{\sqrt{2\pi (1-\rho^2)}}e^{-\frac{(y-\rho x)^2}{2(1-\rho^2)}} \dd y \,\gamma(\dd x)
    = \int_\R x^n \E[M_x^m] \gamma(\dd x),
\end{equation} where $M_x \sim \mathcal{N}(\rho x, 1-\rho^2)$. The moments $\E[M_x^m]$ have the following form:
\begin{gather}
\begin{gathered}\label{eq:M_x_moments}
    \E[M_x]=\rho x, \quad \E[M_x^2]=\rho^2 x^2 + (1-\rho^2), \quad \E[M_x^3]= \rho^3 x^3 +3\rho x(1-\rho^2), \\
    \E[M_x^4]= \rho^4x^4+6\rho^2x^2 (1-\rho^2)+3(1-\rho^2)^2.
\end{gathered}
\end{gather} In the following, we calculate the mixed moments using~\eqref{eq:mixed_moments} and~\eqref{eq:M_x_moments}. Note that $(\U,\V) \overset{\mathcal{D}}{=} (\V,\U)$, hence $\E[\U^n\V^m]=\E[\U^m\V^n]$, and it follows for all $n \in \N$, that
\begin{equation}
\begin{aligned}\label{eq:mixed_moments_exact}
    \E[\U^n \V] &= \int_\R x^{n+1} \rho \gamma(\dd x)=\E[\U^{n+1}] \rho, \\
    \E[\U^n \V^2]&=\int_\R x^n (\rho^2 x^2+(1-\rho^2))\gamma(\dd x)=\E[\U^{n+2}] \rho^2 +\E[\U^n] (1-\rho^2),\\
    \E[\U^n \V^3]&=\int_\R x^n (\rho^3x^3+3\rho x(1-\rho^2))\gamma(\dd x) = \E[\U^{n+3}] \rho^3+\E[\U^{n+1}]3\rho(1-\rho^2), \\
    \E[\U^n \V^4]&=\int_\R x^n (\rho^4x^4+6\rho^2x^2 (1-\rho^2)+3(1-\rho^2)^2)\gamma(\dd x)\\
    &=\E[\U^{n+4}]\rho^4+\E[\U^{n+2}]6\rho^2(1-\rho^2)+\E[\U^n]3(1-\rho^2)^2. 
\end{aligned}
\end{equation} We note that $\E[\U^{2n}]=\E[\V^{2n}]=(2n-1)!!=(2n-1)(2n-3)\cdots 3\cdot 1$, for all $n \in \N$, where $n!!$ is the double factorial. Using~\eqref{eq:mixed_moments_exact}, it holds that
\begin{align*}
    \E[\X^4]&=\E[10000 \U^4 + 4000 \U^3 \V^3 - 12000 \U^3 \V + 600 \U^2 \V^6 - 3600 \U^2 \V^4 \\
    &\qquad + 5400 \U^2 \V^2 + 40 \U \V^9 - 360 \U \V^7 + 1080 \U \V^5\\
    &\qquad  - 1080 \U \V^3 + \V^{12} - 12 \V^{10} + 54 \V^8 - 108 \V^6 + 81 \V^4]\\
    &= 36948 + 12960 \rho + 21600 \rho^2 + 24000 \rho^3.
\end{align*} Hence, since $3\E[\X^2]^2=33708$, it follows that $\kappa_4(\X)=3240 + 12960 \rho + 21600 \rho^2 + 24000 \rho^3$, and
\begin{equation}\label{eq:closed_form_rho}
    \kappa_4(\X)=
    0 \quad \text{ if and only if } \quad \rho= \frac{3}{10} \bigg(\left(\frac{\sqrt{5} - 1}{2}\right)^{1/3}-1 - \left(\frac{2}{\sqrt{5}-1}\right)^{1/3} \bigg).
\end{equation} Above, we used that $\kappa_4(\X)=0$ only has one real root $\rho=-0.39665...$. To show that $\X$ (with $\rho$ as in~\eqref{eq:closed_form_rho}) is not Gaussian, it suffices to show that $\E[\X^6] \ne 15 \E[\X^2]^3=17865240$. Hence, we use~\eqref{eq:mixed_moments_exact} to calculate $\E[\X^6]$:
\begin{align*}
    \E[\X^6]&=\E\big[1000000 \U^6 + 600000 \U^5 \V^3 - 1800000 \U^5 \V + 150000 \U^4 \V^6  \\
    &\qquad - 900000 \U^4 \V^4 + 1350000 \U^4 \V^2 + 20000 \U^3 \V^9 - 180000 \U^3 \V^7   \\
    &\qquad + 540000 \U^3 \V^5 - 540000 \U^3 \V^3 + 1500 \U^2 \V^{12} - 18000 \U^2 \V^{10} + 81000 \U^2 \V^8   \\
    &\qquad - 162000 \U^2 \V^6 + 121500 \U^2 \V^4 + 60 \U \V^{15} - 900 \U \V^{13} + 5400 \U \V^{11}   \\
    &\qquad  - 16200 \U \V^9+ 24300 \U \V^7  - 14580 \U \V^5 + \V^{18} - 18 \V^{16}\\ 
    & \qquad + 135 \V^{14} - 540 \V^{12}  + 1215 \V^{10} - 1458 \V^8 + 729 \V^6\big]\\
    &= 32400000 \rho^4+102960000 \rho^3+104328000 \rho^2+62596800 \rho+34330920 \\
    & = 20292574.8838.... 
\end{align*} This concludes that $\E[\X^6] \ne 15 \E[\X^2]^3$, and hence the proof of part~(ii).
    \end{proof}

    \begin{proof}[Proof of Theorem~\ref{thm: ikke gaussisk1}]
        By assumption, there exist vectors $u \in \H^{\odot p}, v \in \H^{\odot q}$ such that $\X = I_p(u)+I_q(v)$. Since $\X\ne 0$, it holds that $\kappa_4(I_p(u))\vee \kappa_4(I_q(v)) > 0$ due to \cite[Cor. 5.2.11]{MR2962301}. Hence, it follows from Lemma \ref{lem: key lemma} that $\kappa_4(\X) > 0$.
    \end{proof} 

    \begin{proof}[Proof of Proposition~\ref{prop:modeks_p_modeks}]
        Let $(\U,\V)^\intercal \sim \mathcal{N}_2\left(\bm{0},\bm{\Sigma}\right)$, where $\bm{\Sigma}_{1,1}=\bm{\Sigma}_{2,2}=1$ and $\bm{\Sigma}_{1,2}=\bm{\Sigma}_{2,1}=\rho \in [-1,1]$. Then $\ZZ \coloneqq a \U \in \mathbb{H}_1$ and $\Y \coloneqq \V^5-10\V^3+15 \V \in \mathbb{H}_5$, and consider $\X \coloneqq \Y + \ZZ = a \U+\V^5-10\V^3+15 \V$. We show for any $a \in \R$ that there does not exist any real $\rho \in [-1,1]$ such that $\kappa_4(\X)=0$. To calculate $\kappa_4(\X)=\E[\X^4]-3\E[\X^2]^2>0$, note that
\begin{align*}
    \E[\X^2]&=\E[(a \U)^2+(\V^5-10\V^3+15 \V)^2+1000 \U(\V^5-10\V^3+15 \V)]\\
    &=a^2\E[\U^2]+\E[(\V^5-10\V^3+15 \V)^2]=a^2+5!,
\end{align*} 
since $\E[(\V^5-10\V^3+15 \V)^2]=
5!$ and $\E[\U(\V^5-10\V^3+15 \V)]=0$ by~\cite[Prop.~2.2.1]{MR2962301}. Hence, we have $3\E[\X^2]^2=3(a^2+5!)^2$. Next, to calculate $\E[\X^4]$, note that
\begin{align*}
    \E[\X^4]&=\E[a^4 \U^4 + 4 a^3 \U^3 \V^5 - 40 a^3 \U^3 \V^3 + 60 a^3 \U^3 \V + 6 a^2 \U^2 \V^{10} - 120 a^2 \U^2 \V^8 \\
    &\qquad + 780 a^2 \U^2 \V^6 - 1800 a^2 \U^2 \V^4 + 1350 a^2 \U^2 \V^2 + 4 a \U \V^{15} - 120 a \U \V^{13}  \\
    &\qquad + 1380 a \U \V^{11} - 7600 a \U \V^9 + 20700 a \U \V^7 - 27000 a \U \V^5 + 13500 a \U \V^3   \\
    &\qquad + \V^{20} - 40 \V^{18} + 660 \V^{16} - 5800 \V^{14} + 29350 \V^{12} \\
    & \qquad - 87000 \V^{10} + 148500 \V^8 - 135000 \V^6 + 50625 \V^4].
\end{align*} Recall that $\E[\U^{2n}]=\E[\V^{2n}]=(2n-1)!!$ for all $n \in \N$. Then it holds that 
\begin{multline*}
        \E[a^4\U^4+\V^{20} - 40 \V^{18} + 660 \V^{16} - 5800 \V^{14} + 29350 \V^{12} - 87000 \V^{10} \\
        + 148500 \V^8 - 135000 \V^6 + 50625 \V^4]=3a^4+67003200.
    \end{multline*}
 Next, using~\eqref{eq:mixed_moments_exact}, it follows that
    \begin{align}
        & \E[4 a^3 \U^3 \V^5 - 40 a^3 \U^3 \V^3]\\
        &\qquad =4a^3 (105\rho^3+45 \rho(1-\rho^2))-40a^3(15\rho^3+9\rho(1-\rho^2)) =-180 a^3 \rho,
    \end{align}
     
     and similarly
    \begin{align*}
        &\E[60 a^3 \U^3 \V + 6 a^2 \U^2 \V^{10} - 120 a^2 \U^2 \V^8 + 780 a^2 \U^2 \V^6 - 1800 a^2 \U^2 \V^4 + 1350 a^2 \U^2 \V^2  \\
        & \qquad + 4 a \U \V^{15} - 120 a \U \V^{13} + 1380 a \U \V^{11} - 7600 a \U \V^9 + 20700 a \U \V^7 ]\\
        &= 180a^3 \rho + 97920 a^2 \rho^2+12060 a^2+864000a \rho.
    \end{align*} Hence, altogether
    \begin{equation*}
        \E[\X^4]= 97920 a^2 \rho^2+864000 a \rho+3a^4+12060a^2+67003200.
    \end{equation*} 
    Since $3\E[\X^2]^2=3(a^2+5!)^2$, we can conclude for all $a \in \R$, that
    \begin{equation}\label{eq:fourth_cumul_adep}
        \kappa_4(\X)=97920a^2 \rho^2+864000a\rho+11340a^2 +66960000.
    \end{equation} The cumulant in~\eqref{eq:fourth_cumul_adep} only has imaginary roots if $a \ne 0$ and no roots whenever $a=0$. The imaginary roots are given by the ensuing expression when $a \ne 0$:
    \begin{equation*}
        \rho = \frac{\pm \sqrt{6} \sqrt{-a^2 (357 a^2 + 2048000)} - 600 a}{136 a^2}.
    \end{equation*} Hence, we may conclude from~\eqref{eq:fourth_cumul_adep} that there for any $a \in \R$ does not exist any $\rho \in [-1,1]$ such that $\kappa_4(\X)=0$. Thus, since $\kappa_4(\X)>0$ obviously holds for $a>0,\rho\in (0,1]$, it holds that $\kappa_4(\X)>0$ for for all $a\in \R$ and $\rho\in [-1,1]$, since $(a,\rho)\mapsto \kappa_4(\X)$ is a continuous function.
    \end{proof}

    \begin{proof}[Proof of Theorem~\ref{thm: sum uafh.}]
        By rescaling, we may assume without loss of generality that $\sigma^2 = 1$. Since the sequence of fourth moments $(\E[\X_n^4])_{n\in \N}$ is bounded, the sequence $(\X_n)_{n\in \N}$ is tight. By Prokhorov's theorem~\cite[Thm~3.1]
        {billing2}, it suffices to show that if $(\X_{n_k})_{k\in \N}$ is any subsequence which has a limit $\mu$ in distribution, then necessarily $\mu = \mathcal{N}(0,1)$.
        To ease notation, we may assume $\mathsf{X}_n \xrightarrow[]{\mathcal{D}} \mu$ as $n\to \infty$ for some probability measure $\mu$. By assumption, the sequence $(\E[\mathsf{F}_{p,n}^2])_{n\in \N}$ is in the compact set $[0,1]$ for all $p \in \N$. Hence, by Bolzano-Weierstrass~\cite[Thm 3.6]{Rudin}, for all $p \in \N$ there is a $\sigma_p^2$ and a subsequence $(n_{p,k})_{k\in \N}$ such that $\E[\mathsf{F}_{p,n_{p,k}}^2] \to \sigma_p^2$ as $k\to \infty$ and $\sum_{p=1}^\infty \sigma_p^2 = 1$. By defining the \textit{diagonal} subsequence $(n_k)_{k\in \N} \coloneqq (n_{k,k})_{k\in \N}$, we have a sequence which is a subsequence of $(n_{p,k})_{k\in \N}$ for all $p\in \N$.
        Along this subsequence, all the variances converge, i.e., $\E\left[\mathsf{F}_{p,n_{k}}^2\right] \to \sigma_p^2$ for all $p \in \N$ as $k \to \infty$. 
        Using that cumulants are additive when the variables are independent (see e.g.\ ~\cite[p. 148]{Billingsley}) and that sum $\sum_{p=1}^\infty \mathsf{F}_{p,n_k}$ converges in $\mathsf{L}^{\!4}(\Omega)$ by assumption, we have
        \begin{align}
        \begin{aligned}\label{eq: kumulanter}
            0 & = \lim_{k\to \infty}\kappa_4(\hspace{0.3mm}\mathsf{X}_{n_k}) = \lim_{k\to \infty}\lim_{M\to \infty} \kappa_4\Bigg(\sum_{p=1}^M \mathsf{F}_{p,n_k} \Bigg) \\
             &= \lim_{k\to \infty}\lim_{M\to \infty} \sum_{p=1}^M\kappa_4( \mathsf{F}_{p,n_k} ) = \lim_{k\to \infty} \sum_{p=1}^\infty \kappa_4( \mathsf{F}_{p,n_k} ). 
        \end{aligned}
        \end{align}
        By~\cite[Cor.~5.2.11]{MR2962301}, $\kappa_4(\mathsf{F}_{p,n_k})\ge 0$ for all $p,k\in \N$, and hence~\eqref{eq: kumulanter} implies that $\lim_{k\to \infty}\kappa_4\lp \mathsf{F}_{p,n_k}\rp =0$ for all $p \in \N$. Consequently,
        \[
        0 = \lim_{k\to \infty} \kappa_4\lp \mathsf{F}_{p,{n_{k}}}\rp = \lim_{k\to \infty }\lp \E\left[\mathsf{F}_{p,{n_{k}}}^4\right]- 3\E\left[\mathsf{F}_{p,{n_{k}}}^2\right]^2 \rp = 
        \lim_{k\to \infty }\E\left[\mathsf{F}_{p,{n_{k}}}^4\right] - 3\sigma_p^4,
        \]
        or equivalently $\E[\lp \mathsf{F}_{p,{n_{k}}}/\sigma_p\rp^4] \to 3$ as $k\to \infty$. Hence, the fourth-moment theorem applies to the sequences $(\mathsf{F}_{p,n_{k}}/\sigma_p)_{k\in \N}$, yielding for all $p\in \N$ that $\mathsf{F}_{p,{n_{k}}} \xrightarrow[]{\mathcal{D}} \mathcal{N}(0, \sigma_p^2)$ as $k\to \infty$ and by independence $\sum_{p=1}^M \mathsf{F}_{p,n_k} \xrightarrow[]{\mathcal{D}} \mathcal{N}\big( 0,\sum_{p=1}^M \sigma_p^2 \big)$ as $k\to \infty$ for all $M \in \N$. 

        Now let $Z_p$ be independent $ \mathcal{N}(0,\sigma_p^2)$ variables for all $p\in \N$. 
        We conclude the proof by showing pointwise convergence of the associated characteristic functions. Recall that the characteristic function of the standard normal is given as $t\mapsto e^{-t^2/2}$, and let $\epsilon>0$ and $t\in \R$. By the triangle inequality, we have for every $M\in \N$ that
        \begin{align}
            &\left|\E\left[e^{it\mathsf{X}_{n_k}}\right] - e^{-t^2/2} \right| \notag \\
            & \le \left|\E\left[e^{it\mathsf{X}_{n_k}}\right] - \E\left[e^{it \sum_{p=1}^M \mathsf{F}_{p,n_k}}\right] \right| 
            + \left|\E\left[e^{it \sum_{p=1}^M \mathsf{F}_{p,n_k}}\right] - \E\left[e^{it \sum_{p=1}^M Z_p}\right] \right| \label{eq: tre led1} \\
            & \quad + \left| e^{-t^2/2} - \E\left[e^{it \sum_{p=1}^M Z_p}\right] \right|.\label{eq: tre led2}
        \end{align}
        Since $\sum_{p=1}^M Z_p \xrightarrow[]{\mathcal{D}} \mathcal{N}(0,1)$ as $M\to \infty$, we can choose $M'$ large enough such that the term in~\eqref{eq: tre led2} is smaller than $\epsilon/3$ for $M\ge M'$. For the first term in~\eqref{eq: tre led1}, we use that for all $a,b\in \R$,
        \begin{equation}
            \left|e^{ita} - e^{itb} \right| 
            = 
            \left| \int_{ta}^{tb} e^{is} \, \dd s \right| \le \int_{at\wedge bt}^{at\vee bt}\, \dd s =|t||b-a|.
         \end{equation}
         Therefore,
         \begin{align}
            \bigg|\E\left[e^{it\mathsf{X}_{n_k}}\right] &- \E\left[e^{it \sum_{p=1}^M \mathsf{F}_{p,n_k}}\right] \bigg| \le \sup_{k\in \N}\left|\E\left[e^{it\mathsf{X}_{n_k}}\right] - \E\left[e^{it \sum_{p=1}^M \mathsf{F}_{p,n_k}}\right] \notag\right| \notag 
              \\
            &\le \sup_{k\in \N} |t| \E\Bigg[\bigg|\mathsf{X}_{n_k} - \sum_{p=1}^{M}\mathsf{F}_{p,n_k}\bigg|\Bigg] \le \lp\sup_{k\in \N} |t|  \E\Bigg[\bigg|\sum_{p=M+1}^{\infty} \!\mathsf{F}_{p,n_k}\bigg|^2\Bigg]\rp^{1/2} \\
            & \le \lp\sup_{k\in \N} |t|\sum_{p=M+1}^{\infty} \! \E\big[\mathsf{F}_{p,n_k}^2\big]\rp^{1/2}. \label{eq: halesum}
        \end{align}
         Now since $(\hspace{0.3mm}\mathsf{X}_n)_{n\in \N}$ is bounded in $\dom(L)$ we have by definition that 
         \[
         C \coloneqq \sup_{n\in \N} \sum_{p=1}^\infty p^2 \E[\mathsf{F}_{p,n}^2] < \infty. \]
        In particular, for every $p,n\in \N$ we have $\E[\mathsf{F}_{p,n}^2] \le C/p^2$, and thus
        \begin{equation}
            \sup_{n \in \N} \sum_{p=M}^\infty \E[\mathsf{F}_{p,n}^2] \le C \sum_{p=M}^\infty \frac{1}{p^2} \to 0, \quad \text{ as } M\to \infty.  \label{eq: halesummer}
        \end{equation}
        Using~\eqref{eq: halesummer}, there is $M''>0$ such that the upper bound in~\eqref{eq: halesum} is smaller than $\epsilon/3$ for $M\ge M''$. 
        For the fixed number $M^*=M'\vee M''$, we have the convergence $\sum_{p=1}^{M^*} \mathsf{F}_{p,n_k} \xrightarrow[]{\mathcal{D}} \sum_{p=1}^{M^*}Z_p$ as $k\to \infty$. We can then choose a $k'$ (depending on $M^*$) such that the second term in~\eqref{eq: tre led1} is smaller than $\epsilon/3$ for $k\ge k'$, meaning that each term in~\eqref{eq: tre led1} and~\eqref{eq: tre led2} is bounded by $\epsilon/3$ for such $k$. Finally, when $k\ge k'$,
        \[
        \left|\E\left[e^{it\mathsf{X}_{n_k}}\right] - e^{-t^2/2} \right| \le \epsilon,
        \]
        which shows convergence of the characteristic functions for all $t\in \R$. Hence, $\mathsf{X}_{n_k} \xrightarrow[]{\mathcal{D}} \mathcal{N}(0,1)$ as $k \to \infty$, implying that $\mu = \mathcal{N}(0,1)$, which completes the proof.
    \end{proof}

    \begin{proof}[Proof of Corollary~\ref{cor: cor}]
       By rescaling, we may assume without loss of generality that $\sigma^2 = 1$. First, assume that $\E[\X_n^4]\to 3$ as $n\to \infty$. For each $n\in \N$, 
        \[
        \|L\X_n\|_{\L^{\!2}(\Omega)}^2 = \sum_{p=1}^M p^2\|\mathsf{F}_{p,n}\|_{\L^{\!2}(\Omega)}^2 \le \sum_{p=1}^M p^2 <\infty.
        \]
        Therefore, all the assumptions of Theorem~\ref{thm: sum uafh.} are satisfied, and the result follows. 
        Now, assume that $\X_n \xrightarrow[]{\mathcal{D}} \mathcal{N}(0,1)$. By~\cite[Thm 3.5]{billing2} it suffices to show that $(\X_n^4)_{n\in \N}$ is uniformly integrable, which in turn follows if there exists some $k >4$ such that $\sup_{n\in \N} \E[|\X_n|^k]^{1/k}<\infty$ ~\cite[p. 218]{Billingsley}. By \textit{hyper contractivity} of multiple integrals (\cite[Thm 2.7.2]{MR2962301}), there are constants $c_1,c_2,\dots, c_M>0$ such that $\sup_{n\in \N} \E[\mathsf{F}_{p,n}^6]^{1/6} \le c_p\sup_{n\in \N}\E[\mathsf{F}_{p,n}^2]^{1/2} < \infty$ for all $p=1,\dots,M$. Then, by Minkowski's inequality~\cite[p. 242]{Billingsley} $\sup_{n\in \N}\E[\X_n^6]^{1/6} \le \sum_{p=1}^M \sup_{n\in \N} \E[\mathsf{F}_{p,n}^6]^{1/6} < \infty$.
    \end{proof}

\thanks{
\noindent ABO and DKB are supported by AUFF NOVA grant AUFF-E-2022-9-39. CS is supported by the European Union (ERC, TUCLA, 101125203). Views and opinions expressed are, however, those of the author(s) only and do not necessarily reflect those of the European Union or the European Research Council. Neither the European Union nor the granting
authority can be held responsible for them.}

\printbibliography

\end{document}